\documentclass[11pt]{article}

\usepackage{amssymb}
\usepackage[english]{babel}
\usepackage{graphicx,url}

\newtheorem{theo}{Theorem}[section]
\newtheorem{propo}[theo]{Proposition}
\newtheorem{lema}[theo]{Lemma}

\newtheorem{coro}[theo]{Corollary}

\newtheorem{problem}{Problem}

\input amssym.def
\newsymbol\rtimes 226F
\newfont{\nset}{msbm10}
\newcommand{\ns}[1]{\mbox{\nset #1}}

\def\A{{\mbox {\boldmath $A$}}}
\def\B{{\mbox {\boldmath $B$}}}

\def\E{{\mbox {\boldmath $E$}}}
\def\G{\Gamma}

\def\I{{\mbox {\boldmath $I$}}}
\def\J{{\mbox {\boldmath $J$}}}

\def\Ei{{\cal E}}

\def\O{{\mbox {\boldmath $O$}}}

\def\R{\ns{R}}

\def\A{{\mbox {\boldmath $A$}}}
\def\matrixA{{\mbox {\boldmath $A$}}}
\def\calA{{\cal{A}}}
\def\matrixI{{\mbox {\boldmath $I$}}}
\def\matrixJ{{\mbox {\boldmath $J$}}}
\def\O{{\mbox {\boldmath $O$}}}
\def\matrix0{{\mbox {\boldmath $O$}}}

\def\matrixR{\mbox{\boldmath $R$}}
\def\matrixS{\mbox{\boldmath $S$}}
\def\matrixX{\mbox{\boldmath $X$}}
\def\matrixY{\mbox{\boldmath $Y$}}

\def\e{{\mbox{\boldmath $e$}}}

\def\j{{\mbox{\boldmath $j$}}}

\def\p{{\mbox{\boldmath $p$}}}

\def\vec0{\mbox{\bf 0}}

\def\dist{\mathop{\partial }\nolimits}

\def\ecc{\mathop{\rm ecc }\nolimits}

\def\Ker{\mathop{\rm Ker }\nolimits}
\def\tr{\mathop{\rm tr }\nolimits}
\def\som{\mathop{\rm sum }\nolimits}

\def\sp{\mathop{\rm sp }\nolimits}
\def\span{\mathop{\rm span }\nolimits}

\begin{document}

\title{On Almost Distance-Regular Graphs
\thanks{This version is published in Journal of Combinatorial Theory, Series A 118 (2011),
1094-1113. Research supported by the Ministerio de Educaci\'on y Ciencia,
Spain, and the European Regional Development Fund under project
MTM2008-06620-C03-01 and by the Catalan Research Council under project
2009SGR1387.}}

\author{C. Dalf\'{o}$^\dag$, E.R. van Dam$^\ddag$, M.A. Fiol$^\dag$, E.
Garriga$^\dag$, B.L. Gorissen$^{\ddag}$
\\ \\
{\small $^\dag$Universitat Polit\`ecnica de Catalunya, Dept. de Matem\`atica Aplicada IV} \\
{\small Barcelona, Catalonia} {\small (e-mails: {\tt
\{cdalfo,fiol,egarriga\}@ma4.upc.edu})} \\
{\small $^\ddag$Tilburg University, Dept. Econometrics and O.R.} \\
{\small Tilburg, The Netherlands} {\small (e-mails: {\tt
\{edwin.vandam,b.l.gorissen\}@uvt.nl})} \\
}

\date{}

\maketitle

\noindent Keywords: Distance-regular graph; Walk-regular graph; Eigenvalues; Local
multiplicities; Predistance polynomial

\noindent 2010 Mathematics Subject Classification: 05E30, 05C50

\begin{abstract}
\noindent Distance-regular graphs are a key concept in
Algebraic Combinatorics and have given rise to several
generalizations, such as association schemes. Motivated by
spectral and other algebraic characterizations of
distance-regular graphs, we study `almost distance-regular
graphs'. We use this name informally for graphs that share some
regularity properties that are related to distance in the
graph. For example, a known characterization of a
distance-regular graph is the invariance of the number of walks
of given length between vertices at a given distance, while a
graph is called walk-regular if the number of closed walks of
given length rooted at any given vertex is a constant. One of
the concepts studied here is a generalization of both
distance-regularity and walk-regularity called
$m$-walk-regularity. Another studied concept is that of
$m$-partial distance-regularity or, informally,
distance-regularity up to distance $m$. Using eigenvalues of
graphs and the predistance polynomials, we discuss and relate
these and other concepts of almost distance-regularity, such as
their common generalization of $(\ell,m)$-walk-regularity. We
introduce the concepts of punctual distance-regularity and
punctual walk-regularity as a fundament upon which almost
distance-regular graphs are built. We provide examples that are
mostly taken from the Foster census, a collection of symmetric
cubic graphs. Two problems are posed that are related to the
question of when almost distance-regular becomes whole
distance-regular. We also give several characterizations of
punctually distance-regular graphs that are generalizations of
the spectral excess theorem.
\end{abstract}

\section{Introduction}
Distance-regular graphs \cite{bcn} are a key concept in
Algebraic Combinatorics \cite{g93} and have given rise to
several generalizations, such as association schemes
\cite{mtanaka}. Motivated by spectral \cite{vdhks06} and other
algebraic \cite{fiol02} characterizations of distance-regular
graphs, we study `almost distance-regular graphs'. We use this
name informally for graphs that share some regularity
properties that are related to distance in the graph. For
example, a known characterization (by Rowlinson~\cite{r97}) of
a distance-regular graph is the invariance of the number of
walks of given length between vertices at a given distance.
Godsil and McKay \cite{gmk} called a graph walk-regular if the
number of closed walks of given length rooted at any given
vertex is a constant, cf. \cite[p. 86]{g93}. One of the
concepts studied here is a generalization of both
distance-regularity and walk-regularity called
$m$-walk-regularity, as introduced in \cite{DaFiGa09}. Another
studied concept is that of $m$-partial distance-regularity or,
informally, distance-regularity up to distance $m$. Formally,
it means that for $i \le m$, the distance-$i$ matrix can be
expressed as a polynomial of degree $i$ in the adjacency
matrix. Related to this, there are two other generalizations of
distance-regular graphs. Weichsel \cite{w82} introduced
distance-polynomial graphs as those graphs for which each
distance-$i$ matrix can be expressed as a polynomial in the
adjacency matrix. Such graphs were also studied by Beezer
\cite{beezer}. A graph is called distance degree regular if
each distance-$i$ graph is regular. Such graphs were studied by
Bloom, Quintas, and Kennedy \cite{bloom},  Hilano and Nomura
\cite {hilanomura}, and also by Weichsel \cite{w82} (as
super-regular graphs).

This paper is organized as follows. In the next section we give
the basic background for our paper. This includes our two main
tools: eigenvalues of graphs and their predistance polynomials.
In Section \ref{sec:3}, we discuss several concepts of almost
distance-regularity, such as partial distance-regularity in
Section \ref{sec:3.2} and $m$-walk-regularity in Section
\ref{sec:3.4}. These concepts come together in Section
\ref{sec:lmwr}, where we discuss $(\ell,m)$-walk-regular
graphs, as introduced in \cite{DaFiGa09b}. Sections
\ref{sec:pdr} and \ref{sec:3.3} are used to introduce the
concepts of punctual distance-regularity and punctual
walk-regularity. These form the fundament upon which almost
distance-regular graphs are built. Illustrating examples are
mostly taken from the Foster census \cite{rcmd09}, a collection
of symmetric cubic graphs that we checked by computer for
almost distance-regularity. In Section \ref{sec:3} we also pose
two problems. Both are related to the question of when almost
distance-regular becomes whole distance-regular. The spectral
excess theorem \cite{fg2} is also of this type: it states that
a graph is distance-regular if for each vertex, the number of
vertices at extremal distance is the right one (i.e., some
expression in terms of the eigenvalues), cf. \cite{vd08,fgg09}.
In Section \ref{sec:4} we give several characterizations of
punctually distance-regular graphs that have the same flavor as
the spectral excess theorem. We will show in Section
\ref{specmaxdiam} that these results are in fact
generalizations of the spectral excess theorem. In this final
section we focus on the case of graphs with spectrally maximum
diameter (distance-regular graphs are such graphs).

\section{Preliminaries}
In this section we give the background on which our study is
based. We would like to stress that in this paper we restrict
to simple, connected, and regular graphs, unless we explicitly
state otherwise. First, let us recall some basic concepts
and define our generic notation for graphs.

\subsection{Spectra of graphs and walk-regularity}
Throughout this paper, $\G=(V,E)$ denotes a simple, connected,
$\delta$-regular graph, with order $n=|V|$ and adjacency matrix
$\A$. The {\it distance} between two vertices $u$ and $v$ is
denoted by $\partial (u,v)$, so that the {\it eccentricity} of
a vertex $u$ is $\ecc(u)=\max_{v\in V}\dist (u,v)$ and the {\it
diameter} of the graph is $D=\max_{u\in V}\ecc(u)$. The set of
vertices at distance $i$, from a given vertex $u\in V$ is
denoted by $\Gamma_i(u)$, for $i=0,1,\dots,D$. The degree of a
vertex $u$ is denoted by $\delta(u)=|\Gamma_1(u)|$. The {\em
distance-$i$ graph} $\G_i$ is the graph with vertex set $V$ and
where two vertices $u$ and $v$ are adjacent if and only if
$\dist(u,v)=i$ in $\G$. Its adjacency matrix $\A_i$ is usually
referred to as the {\em distance-$i$ matrix} of $\G$. The
spectrum of $\G$ is denoted by
$$
\sp \G = \sp \A = \{\lambda_0^{m_0},\lambda_1^{m_1},\dots,
\lambda_d^{m_d}\},
$$
where the different eigenvalues of $\G$ are in decreasing order,
$\lambda_0>\lambda_1>\cdots >\lambda_d$, and the superscripts
stand for their multiplicities $m_i=m(\lambda_i)$. In
particular, note that $\lambda_0=\delta$, $m_0=1$ (since $\G$ is
$\delta$-regular and connected) and $m_0+m_1+\cdots+m_d=n$.

For a given ordering of the vertices of $\G$, the vector space
of linear combinations (with real coefficients) of the vertices
is identified with $\R^n$, with canonical basis $\{\e_u : u\in
V\}$. Let $Z=\prod_{i=0}^d (x-\lambda_i)$ be the minimal
polynomial of $\A$. The vector space $\R_d[x]$ of real
polynomials of degree at most $d$ is isomorphic to $\R[x]/(Z)$.
For every $i=0,1,\dots,d$, the orthogonal projection of $\R^n$
onto the eigenspace $\Ei_i=\Ker (\A-\lambda_i \I)$ is given by
the Lagrange interpolating polynomial
$$
\lambda_i^*=\frac{1}{\phi_i}\prod_{\stackrel{j=0}{j\neq i}}^d
(x-\lambda_j) =\frac{(-1)^i}{\pi_i}\prod_{\stackrel{j=0}{j\neq i}}^d
(x-\lambda_j)
$$
of degree $d$, where $\phi_i=\prod_{j=0,j\neq i}^d
(\lambda_i-\lambda_j)$ and $\pi_i=|\phi_i|$. These polynomials
satisfy $\lambda_i^*(\lambda_j)=\delta_{ij}$. The matrices
$\E_i=\lambda_i^*(\A)$, corresponding to these orthogonal
projections, are the {\it $($principal\/$)$ idempotents} of
$\A$, and are known to satisfy the properties:
$\E_i\E_j=\delta_{ij}\E_i$; $\A\E_i=\lambda_i\E_i$; and
$p(\A)=\sum_{i=0}^d p(\lambda_i)\E_i$, for any polynomial $p\in
\R[x]$ (see e.g. Godsil~\cite[p. 28]{g93}). The {\em
$(u$-$)$local multiplicities} of the eigenvalue $\lambda_i$ are
defined as
$$
m_u(\lambda_i) = \|\E_i\e_u\|^2
= \langle\E_i\e_u,\e_u\rangle =
(\E_i)_{uu}\qquad (u\in V;\ i =0,1,\dots,d),
$$
and satisfy $\sum_{i=0}^d m_u(\lambda_i) = 1$ and $\sum_{u\in
V} m_u(\lambda_i) =m_i$, $i=0,1,\dots,d$ (see Fiol and Garriga
\cite{fg2}).

Related to this concept, we say that $\G$ is {\em
spectrum-regular} if, for any $i=0,1,\ldots, d$, the $u$-local
multiplicity of $\lambda_i$ does not depend on the vertex $u$.
Then, the above equations imply that the (standard)
multiplicity `splits' equitably among the $n$ vertices, giving
$m_u(\lambda_i)=m_i/n$.

By analogy with the local multiplicities, which correspond to
the diagonal entries of the idempotents, Fiol, Garriga, and
Yebra~\cite{fgy99} defined the {\it crossed $(uv$-$)$local
multiplicities\/} of the eigenvalue $\lambda_i$,  denoted by
$m_{uv}(\lambda_i)$, as
$$
m_{uv}(\lambda_i)=\langle\E_i\e_u,\E_i\e_v\rangle
=\langle\E_i\e_u,\e_v\rangle=(\E_i)_{uv} \qquad (u,v\in V;\ i =0,1,\dots,d).
$$
(Thus, in particular, $m_{uu}(\lambda_i)=m_{u}(\lambda_i)$.)
These parameters allow us to compute the number of walks of
length $\ell$ between two vertices $u,v$ in the following way:
\begin{equation}
\label{crossed-mul->num-walks} a_{uv}^{({\ell})} =(\A^{\ell})_{uv}=
\sum_{i=0}^d m_{uv}(\lambda_i)\lambda_i^{\ell} \qquad (\ell =0,1,\dots).
\end{equation}
Conversely, given the eigenvalues from which we compute the
polynomials $\lambda_i^*$, and the tuple
${\cal{C}}_{uv}=(a_{uv}^{(0)},a_{uv}^{(1)},\ldots,a_{uv}^{(d)})$,
we can obtain the crossed local multiplicities. With this aim,
let us introduce the following notation: given a polynomial
$p=\sum_{i=0}^d \zeta_i x^i$, let
$p({\cal{C}}_{uv})=\sum_{i=0}^d \zeta_i a_{uv}^{(i)}$. Thus,
\begin{equation}
\label{num-walks->crossed-mul}
 m_{uv}(\lambda_i) = (\E_i)_{uv} =
(\lambda_i^*(\A))_{uv} = \lambda_i^*({\cal{C}}_{uv})  \qquad (i =0,1,\dots,d).
\end{equation}

Let $a_u^{({\ell})}$ denote the number of closed walks of
length ${\ell}$ rooted at vertex $u$, that is,
$a_u^{({\ell})}=a_{uu}^{({\ell})}$. If these numbers only
depend on ${\ell}$, for each $\ell \ge 0$, then $\G$ is called
{\em walk-regular} (a concept introduced by Godsil and
McKay~\cite{gmk}). In this case we write
$a_u^{({\ell})}=a^{({\ell})}$. Notice that, as
$a_u^{(2)}=\delta(u)$, the degree of vertex $u$, a walk-regular
graph is necessarily regular. By (\ref{crossed-mul->num-walks})
and (\ref{num-walks->crossed-mul}) it follows that
spectrum-regularity and walk-regularity are equivalent
concepts. It also shows that the existence of the constants
$a^{(0)},a^{(1)},\ldots,a^{(d)}$ suffices to assure
walk-regularity. It is well known that any distance-regular
graph, as well as any vertex-transitive graph, is walk-regular,
but the converse is not true.

\subsection{The predistance polynomials and distance-regularity}
A graph is called {\em distance-regular} if there are constants
$c_i, a_i,b_i$ such that for any $i=0,1,\dots, D$, and any two
vertices $u$ and $v$ at distance $i$, among the neighbours of
$v$, there are $c_i$ at distance $i-1$ from $u$, $a_i$ at
distance $i$, and $b_i$ at distance $i+1$. In terms of the
distance matrices $\A_i$ this is equivalent to
$$
\A\A_i=b_{i-1}\A_{i-1}+a_i
\A_i+c_{i+1}\A_{i+1}\qquad (i=0,1,\dots,D)
$$
(with $b_{-1}=c_{D+1}=0$). From this recurrence relation, one
can obtain the so-called {\em distance polynomials $p_i$}.
These are such that $\deg p_i=i$ and $\A_i=p_i(\A)$,
$i=0,1,\dots,D$.

From the spectrum of a given (arbitrary, but connected regular) graph, $\sp \G = \linebreak
\{\lambda_0^{m_0},\lambda_1^{m_1},\ldots, \lambda_d^{m_d}\}$,
one can generalize the distance polynomials of a
distance-regular graph by considering the
following scalar product in $\R_d[x]$:
\begin{equation}
\label{product} \langle p, q\rangle =\frac{1}{n}\tr
(p(\A)q(\A))=\frac{1}{n} \sum_{i=0}^d m_i p(\lambda_i) q(\lambda_i).
\end{equation}
Then, by using the Gram-Schmidt method and normalizing
appropriately, it is routine to prove the existence and
uniqueness of an orthogonal system of so-called {\em
predistance polynomials} $\{p_i\}_{0\le i\le d}$ satisfying
$\deg p_i=i$ and $\langle p_i,p_j \rangle=
\delta_{ij}p_i(\lambda_0)$ for any $i,j=0,1,\dots d$. For
details, see Fiol and Garriga~\cite{fg2,fg3}.

As every sequence of orthogonal polynomials, the predistance
polynomials satisfy a three-term recurrence of the form
\begin{equation}
\label{recur-pol}
xp_i=\beta_{i-1}p_{i-1}+\alpha_i p_i+\gamma_{i+1}p_{i+1}\qquad (i=0,1,\dots,d),
\end{equation}
where the constants $\beta_{i-1}$, $\alpha_i$, and
$\gamma_{i+1}$ are the Fourier coefficients of $xp_i$ in terms
of $p_{i-1}$, $p_i$, and $p_{i+1}$, respectively (and
$\beta_{-1}=\gamma_{d+1}=0$), with initial values $p_0=1$ and
$p_1=x$. Let $\omega_k$ be the leading coefficient of $p_k$.
Then, from the above recurrence, it is immediate that
\begin{equation} \label{omega_k}
\omega_k=\frac
1{\gamma_1\gamma_2\cdots \gamma_k}.
\end{equation}
In general, we define the
{\em preintersection numbers} $\xi_{ij}^k$, with
$i,j,k=0,1,\dots d$, as the Fourier coefficients of $p_ip_j$ in
terms of the basis $\{p_k\}_{0\le k\le d}$; that is:
\begin{equation}\label{preintersec}
\xi_{ij}^k=\frac{\langle
p_ip_j,p_k\rangle}{\|p_k\|^2}=\frac{1}{np_k(\lambda_0)}\sum_{l=0}^d
m_lp_i(\lambda_l)p_j(\lambda_l)p_k(\lambda_l).
\end{equation}
With this notation, notice that the constants in
(\ref{recur-pol}) correspond to the preintersection numbers
$\alpha_i=\xi_{1,i}^i$, $\beta_i=\xi_{1,i+1}^i$, and
$\gamma_i=\xi_{1,i-1}^i$. As expected, when $\G$ is
distance-regular, the predistance polynomials and the
preintersection numbers become the distance polynomials and the
{\em intersection numbers} $p_{ij}^k=|\G_i(u)\cap\G_j(v)|$,
$\dist(u,v)=k$, for $i,j,k=0,1,\dots,D (=d)$. For an arbitrary
graph we say that the intersection number $p^k_{ij}$ is {\em
well-defined} if $|\G_i(u)\cap\G_j(v)|$ is the same for all
vertices $u,v$ at distance $k$, and we let $a_i=p_{1,i}^i$,
$b_i=p_{1,i+1}^i$, and $c_i=p_{1,i-1}^i$. From a combinatorial
point of view, we would like many of these intersection numbers
to be well-defined, in order to call a graph almost
distance-regular.

Note that not all properties of the distance polynomials of
distance-regular graphs hold for the predistance polynomials.
The crucial property that is not satisfied in general is that
of the equations $\A_i = p_i(\A)$. In fact, informally speaking
we will `measure' almost distance-regularity by how much the
matrices $\A_i$ look like the matrices $p_i(\A)$. Walk-regular
graphs, for example, were characterized by Dalf\'{o}, Fiol, and
Garriga~\cite{DaFiGa09} as those graphs for which the matrices
$p_i(\A)$, $i=1,\dots, d$, have null diagonals (as have the
matrices $\A_i$, $i=1,\dots, d$).

A property that holds for all connected graphs is that the sum of all
predistance polynomials gives the Hoffman polynomial $H$:
\begin{equation}
\label{polHof}
 H = \sum_{i=0}^d p_i =  \frac{n}{\pi_0} \prod_{i=1}^d
(x-\lambda_i) = n\,\lambda_0^*,
\end{equation}
which characterizes regular graphs by the condition $H(\A)=\J$,
the all-$1$ matrix \cite{hof63}. Note that (\ref{polHof})
implies that $\omega_d=\frac n{\pi_0}$. It can also be used to
show that $\alpha_i+\beta_i+\gamma_i=\lambda_0=\delta$ for all
$i$.

For bipartite graphs we observe the following facts. Because
the eigenvalues are symmetric about zero
($\lambda_i=-\lambda_{d-i}$ and $m_i=m_{d-i}$, $0\le i\le d$),
we have $\langle xp_i,p_i \rangle =0$ from (\ref{product}), and
therefore $\alpha_i=0$ for all $i$. It then follows from
(\ref{recur-pol}) that the predistance polynomials $p_i$ are
even for even $i$, and odd for odd $i$. Using
(\ref{preintersec}), this implies among others that
$\xi^k_{ij}=0$ if $i+j+k$ is odd. It also follows that
$\gamma_d=\lambda_0=\delta$. Finally, the Hoffman polynomial
splits into an even part $H_0=\sum_{i}p_{2i}$ and an odd part
$H_1=H-H_0$, and these have the property that $(H_0)_{uv}=1$ if
$u$ and $v$ are in the same part of the bipartition, and
$(H_1)_{uv}=1$ if $u$ and $v$ are in different parts.

\subsection{The adjacency algebra and the distance algebra
}\label{subsec_alg}

Given a graph $\G$, the set $\calA= \{p(\matrixA):\,
p\in\R[x]\}$ is a vector space of dimension $d+1$ and also an
algebra with the ordinary product of matrices, known as the
{\it adjacency algebra}, and
$\{\matrixI,\matrixA,\ldots,\matrixA^d\}$ is a basis of
$\calA$. Since $\matrixI,\matrixA,\matrixA^2,\ldots,\matrixA^D$
are linearly independent, we have that $\dim \mathcal{A}=d+1\ge
D+1$ and therefore the diameter is at most $d$. A natural
question is to enhance the case when equality is attained; that
is, $D=d$. In this case, we say that the graph $\G$ has {\it
spectrally maximum} diameter.

Let ${\mathcal D}$ be the linear span of the set
$\{\matrixA_0,\matrixA_1,\ldots,\matrixA_D\}$. The
($D+1$)-dimensional vector space ${\mathcal D}$ forms an
algebra with the entrywise or Hadamard product of matrices,
defined by
$(\matrixX\circ\matrixY)_{uv}=\matrixX_{uv}\matrixY_{uv}$. We
call ${\mathcal D}$ the {\em distance $\circ$-algebra}.

In the following sections, we will work with the vector space
${\cal T}={\cal A}+{\cal D}$, and relate the distance-$i$
matrices $\A_i \in {\mathcal D}$ with the matrices $p_i(\A) \in
{\mathcal A}$. Note that $\matrixI$, $\matrixA$, and $\matrixJ$
are matrices in ${\cal A}\cap{\cal D}$ since
$\matrixJ=H(\matrixA)\in \mathcal{A}$. Thus, $\dim ({\cal
A}\cap {\cal D}) \geq 3$, if $\G$ is not a complete graph (in
this exceptional case $\J=\I+\A$). Note that ${\mathcal
A}={\mathcal D}$ if and only if $\G$ is distance-regular, which
is therefore equivalent to $\dim ({\cal A}\cap {\cal D}) =
d+1$. For this reason, the dimension of ${\cal A}\cap {\cal D}$
(compared to $D$ and $d$) can also be seen as a measure of
almost distance-regularity.

One concept of almost distance-regularity related to this was
introduced by Weichsel \cite{w82}: a graph is called {\em
distance-polynomial} if ${\cal D} \subset {\cal A}$, that is,
if each distance matrix is a polynomial in $\A$. Hence a graph
is distance-polynomial if and only if $\dim ({\cal A}\cap {\cal
D})=D+1$.

Note that for any pair of (symmetric) matrices
$\matrixR,\matrixS\in{\cal T}$, we have
\[\tr (\matrixR\matrixS)=
  \sum_{u\in V}(\matrixR\matrixS)_{uu}=
  \sum_{u\in V}\sum_{v\in V}\matrixR_{uv}\matrixS_{vu}=
  \som(\matrixR\circ\matrixS).\]
Thus, we can define a scalar product in ${\cal T}$ in two
equivalent forms:
$$
\langle\matrixR,\matrixS\rangle=
\frac 1n\tr (\matrixR\matrixS)= \frac
1n\som(\matrixR\circ\matrixS).
$$
\noindent In ${\cal A}$, this scalar product coincides with the
scalar product (\ref{product}) in $\R[x]/(Z)$, in the sense
that $\langle p(\A),q(\A)\rangle=\langle p,q\rangle$. Observe
that the factor $1/n$ assures that $\|\matrixI\|^2=\langle
1,1\rangle =1$. Note also that $\|\A_i\|^2=\overline{\delta}_i$
(the {\em average degree} of $\G_i$), whereas
$\|p_i(\A)\|^2=p_i(\lambda_0)$.

Association schemes are generalizations of distance-regular
graphs that will provide almost distance-regular graphs. A
(symmetric) {\em association scheme} can be defined as a set of
symmetric $(0,1)$-matrices (graphs) $\{\B_0=\I, \B_1, \dots,
\B_e\}$ adding up to the all-1 matrix $\J$, and whose linear
span is an algebra ${\cal B}$ (with both --- the ordinary and the Hadamard --- products),
called the Bose-Mesner algebra. In the case of distance-regular
graphs, the distance-matrices $\A_i$ form an association
scheme. For more on association schemes, we refer to a recent
survey by Martin and Tanaka \cite{mtanaka}.

\section{Different concepts of almost
distance-regularity}\label{sec:3} In this section we introduce
some concepts of almost distance-regular graphs, together with
some characterizations. We begin with some closely related
`local concepts' concerning distance-regular and
distance-polynomial graphs.

\subsection{Punctually distance-polynomial and punctually distance-regular \newline
graphs}\label{sec:pdr}

We recall that in this paper $\G$ denotes a connected regular graph.
We say that a graph $\G$ is $h$-{\em punctually
distance-polynomial} for an integer $h\leq D$, if $\A_h\in
{\cal A}$; that is, there exists a polynomial $q_h\in\R_d[x]$
such that $q_h(\A)=\A_h$. Obviously, $\deg q_h \ge h$. In case
of equality, i.e., if $\deg q_h=h$, we call the graph $h$-{\em
punctually distance-regular}. Notice that, since $\A_0=\I$ and
$\A_1=\A$, every graph is $0$-punctually distance-regular
($q_0=1$) and $1$-punctually distance-regular ($q_1=x$). In
general, we have the following result.
\begin{lema}
\label{lema3.1} Let $h \le D$ and let $\G$ be $h$-punctually
distance-polynomial, with $\A_h=q_h(\A)$. Then the distance-$h$
graph $\G_h$ is regular of degree $q_h(\lambda_0)=\|q_h\|^2$. If
$\deg q_h=h$ $($$\G$ is $h$-punctually distance-regular$)$, then
$q_h=p_h$, the predistance polynomial of degree $h$. If $\deg
q_h>h$, then $\deg q_h>D$.
\end{lema}
 \begin{proof}
Let $\j$ denote the all-$1$ vector. Because $\A_h
\j=q_h(\A)\j=q_h(\lambda_0)\j$, the graph $\G_h$ is regular with
degree $q_h(\lambda_0)=\frac{1}{n}\tr (\A_h^2)=
\|\A_h\|^2=\|q_h\|^2$. Moreover, for every polynomial
$p\in\mathbb{R}_{h-1}[x]$, we have $\langle
q_h,p\rangle=\langle \A_h,p(\A)\rangle= 0$. Thus, if $\deg
q_h=h$, we must have $q_h=p_h$ by the uniqueness of the
predistance polynomials. If $h<\deg q_h=i \le D$ and $q_h$ has
leading coefficient $\varsigma_i$ then we would have
$(q_h(\A))_{uv}=\varsigma_i a^{(i)}_{uv}\neq 0$ for any two
vertices $u,v$ at distance $i$, which contradicts
$(q_h(\A))_{uv}=(\A_h)_{uv}=0$.
 \end{proof}

This lemma implies that the concepts of $h$-punctually
distance-polynomial and $h$-punctually distance-regular are the
same for graphs with spectrally maximum diameter $D=d$. We will
consider such graphs in more detail in Section
\ref{specmaxdiam}.

Any polynomial of degree at most $d$ is a linear combination of
the polynomials $p_0,\dots,p_d$. If $\A_h=q_h(\A)$, then
clearly $q_h$ is a linear combination of the polynomials
$p_h,\dots,p_d$. For example, in the case of a graph with $D=2$
(which is always distance-polynomial; see the next section), we
have $\A_2=q_2(\A)$, with $q_2=p_2+\cdots +p_d$.

On the other hand, if $p_h(\A)$ is a linear combination of the
distance-matrices $\A_i, i=0,1,\dots,D$, then we have the
following.
\begin{lema}
\label{lem D to A} Let $h\le d$. If $p_h(\A)\in{\cal D}$, then
 $h \leq D$ and $\G$ is $h$-punctually distance-regular.
\end{lema}
\begin{proof} If $p_h(\A)\in{\cal D}$, then
$p_h(\A)=\sum_{i=0}^h\zeta_i\A_i$ for some $\zeta_i$, $i=0,1,\dots,h$.
Note first that $ \langle \A_i,p_i(\A)\rangle= \frac
1n\sum_{\partial(u,v)=i}(p_i(\A))_{uv}=\frac{\omega_i}{n}
\sum_{\partial(u,v)=i}(\A^i)_{uv}\neq 0$ for $i \le D$. Now it
follows that $0=\langle p_h(\A),p_0(\A)\rangle= \zeta_0\langle
\A_0,p_0(\A)\rangle$ and hence that $\zeta_0=0$. By using that
$0=\langle p_h(\A),p_i(\A)\rangle$ one can similarly show by
induction that $\zeta_i=0$ for $i<h$. If $h>D$, then this
implies that $p_h(\A)=\O$, which is a contradiction. Hence $h
\le D$ and $\A_h=\frac 1{\zeta_h}p_h(\A)$. By Lemma
\ref{lema3.1} it then follows that $\A_h=p_h(\A)$, i.e., that
$\G$ is $h$-punctually distance-regular.
\end{proof}

Graph F026A from the Foster Census \cite{rcmd09} is an example
of a (bipartite) graph with $D=d=5$, that is $h$-punctually
distance-regular for $h=2$ and $4$, but not for $h=3$ and $5$.
It is interesting to observe, however, that the intersection
number $c_5=3$ is well-defined, whereas
$|\G_1(u)\cap\G_3(v)|=2$ or $3$ for $\dist(u,v)=4$, so $c_4$ is
not well-defined. Thus, there does not seem to be a
combinatorial interpretation in terms of intersection numbers
of the algebraic definition of punctual distance-regularity. In
the next section, the combinatorics will return.

\subsection{Partially distance-polynomial and partially distance-regular
graphs}\label{sec:3.2}

A graph $\G$ is called $m$-{\em partially distance-polynomial}
if $\A_h=q_h(\A)\in {\cal A}$ for every $h\le m$ (that is, $\G$
is $h$-punctually distance-polynomial for every $h\le m$). If
each polynomial $q_h$ has degree $h$, for $h \le m$, we call
the graph $m$-{\em partially distance-regular} (that is, $\G$ is
$h$-punctually distance-regular for every $h\le m$). In this
case, $\A_h=p_h(\A)$ for $h \le m$, by Lemma \ref{lema3.1}.

Alternatively, and recalling the combinatorial properties of
distance-regular graphs, we can say that a graph is
$m$-partially distance-regular when the intersection numbers
$c_i$, $a_i$, $b_i$ up to $c_m$ are well-defined, i.e., the
distance matrices satisfy the recurrence
$$
\A\A_i =
b_{i-1}\A_{i-1}+a_i\A_i+c_{i+1}\A_{i+1}\qquad (i=0,1,\dots,
m-1).
$$
From this we have the following lemma, which may be useful in
finding examples of $m$-partially distance-regular graphs with large
$m$.
\begin{lema}\label{girth}
If $\G$ has girth $g$, then $\G$ is $m$-partially
distance-regular with $m=\lfloor \frac{g-1}{2}\rfloor$.
\end{lema}
\begin{proof}
Just note that if the girth is $g$ then there is a unique shortest
path between any two vertices at distance at most $m=\lfloor
\frac{g-1}{2}\rfloor$. Hence the intersection parameters $c_i$,
$b_i$, and $a_i$ up to $c_m$ are well-defined; indeed, if $\G$ has
degree $\delta$, then $c_i=1$, $1\le i\le m$; $a_i=0$, $0\le i\le
m-1$; and $b_0=\delta$, $b_i=\delta-1$, $1\le i\le m-1$.
\end{proof}

Generalized Moore graphs are regular graphs with girth at least
$2D-1$, cf. \cite{ms78, s04}. By Lemma \ref{girth}, such graphs
are $(D-1)$-partially distance-regular. Only few examples of
generalized Moore graphs that are not distance-regular are
known.

It is clear that every $D$-partially distance-polynomial graph is
distance-polynomial, and every $D$-partially distance-regular graph
is distance-regular (in which case $d=D$). In fact, the conditions
can be slightly relaxed as follows.
\begin{propo}
\label{D/d-1 suffices} If $\G$ is $(D-1)$-partially
distance-polynomial, then $\G$ is distance-\linebreak
polynomial. If $\G$ is $(d-1)$-partially distance-regular, then
$\G$ is distance-regular.
\end{propo}
\begin{proof}
Let $\G$ be $(D-1)$-partially distance-polynomial, with
$\A_h=q_h(\A)$, $h\le D-1$. Then by using the expression for
the Hoffman polynomial in (\ref{polHof}), we have:
$$
\A_D +\sum_{h=0}^{D-1}q_h(\A)=
\sum_{h=0}^D\A_h=\J=H(\A),
$$
so that $\A_D=q_D(\A)$, where $q_D=H-\sum_{h=0}^{D-1}q_h$, and
$\G$ is distance-polynomial.

Similarly, if $\G$ is $(d-1)$-partially distance-regular, then
from $\A_d+\sum_{i=0}^{d-1} p_i(\A)=\sum_{i=0}^d\A_i=H(\A)$, we
get $\A_{d}=p_d(\A)$, and $\G$ is distance-regular.
\end{proof}

In particular, Proposition \ref{D/d-1 suffices} implies the
observation by Weichsel \cite{w82} that every (regular) graph
with diameter two is distance-polynomial.

The distinction between $D$ and $d$ in Proposition \ref{D/d-1
suffices} is essential. A $(D-1)$-partially distance-regular
graph is not necessarily distance-regular. In fact, Koolen and
Van Dam [private communication] observed that the direct
product of the folded $(2D-1)$-cube \cite[p. 264]{bcn} and
$K_2$ is $(D-1)$-partially distance-regular with diameter $D$,
but $a_{D-1}$ is not well-defined. Note that these graphs also
occur as so-called boundary graphs in related work
\cite{fgy99}.

It would also be interesting to find examples of $m$-partially
distance-regular graphs with $m$ equal (or close) to $d-2$ that
are not distance-regular (for all $d$), if any exist. More
specifically, we pose the following problem.
\begin{problem}\label{problem1} Determine the smallest $m=m_{pdr}(d)$ such that
every $m$-partially distance-regular graph with $d+1$ distinct
eigenvalues is distance-regular.
\end{problem}

For bipartite graphs, the result in Proposition \ref{D/d-1 suffices}
can be improved as follows.
\begin{propo}
\label{bipD/d-2 suffices} Let $\G$ be bipartite. If $\G$ is
$(D-2)$-partially distance-polynomial, then $\G$ is
distance-polynomial. If $\G$ is $(d-2)$-partially distance-regular,
then $\G$ is distance-regular.
\end{propo}
\begin{proof}
Similar as the proof of Proposition \ref{D/d-1 suffices};
instead of the Hoffman polynomial, one should use its even and
odd parts $H_{0}$ and $H_{1}$.
\end{proof}

It is interesting to note that a graph with $D=d$ that is
$D$-punctually distance-regular must be distance-regular. This
result is a small part in the proof of the spectral excess theorem,
cf. \cite{vd08,fgg09}. We will generalize this in Proposition
\ref{prop pld curta} by showing that we do not need to have
$h$-punctual distance-regularity for all $h\le m$ to obtain
$m$-partial distance-regularity. The following lemma is a first step
in this direction.

\begin{lema}
\label{lem curt a llarg} Let $d-m<s \le m \le D$ and let $\G$ be
$h$-punctually distance-regular for $h=m-s+1,\dots,m$. Then $\G$
is $(m-s)$-punctually distance-regular.
\end{lema}
\begin{proof}
By the assumption, we have $\A_{m-s+1}=p_{m-s+1}(\A)$, $\ldots$ ,
$\A_m=p_m(\A)$, and we want to show that $p_{m-s}(\A)=\A_{m-s}$. We
therefore check the entry $uv$ in $p_{m-s}(\A)$, and distinguish the
following three cases:
\begin{itemize}
  \item[$(a)$] For $\partial(u,v)>m-s$, we have $(p_{m-s}(\A))_{uv}=0$.
  \item[$(b)$] For $\partial(u,v)<m-s$, we use the equation
$xp_{m-s+1}=\beta_{m-s}p_{m-s}+\alpha_{m-s+1}p_{m-s+1}+
            \gamma_{m-s+2}p_{m-s+2},$ which gives us
$\A\A_{m-s+1}=\beta_{m-s}p_{m-s}(\A) + \alpha_{m-s+1}\A_{m-s+1} +
\gamma_{m-s+2}\A_{m-s+2}$ (in case $s=1$ we have $m=d$ and then the
last term vanishes). Hence it follows that
$$(p_{m-s}(\A))_{uv}=\frac1{\beta_{m-s}}(\A\A_{m-s+1})_{uv}=
\frac1{\beta_{m-s}}\sum_{w\in \Gamma_1(u)}(\A_{m-s+1})_{wv}=0,$$
since $\partial(v,w)\le \partial(v,u)+\partial(u,w)<m-s+1$
for the relevant $w$.
  \item[$(c)$] For $\partial(u,v)=m-s$, we claim that
$(p_i(\A))_{uv}=0$ for $i\neq m-s$. This is clear if $i<m-s$ and
also if $m-s+1\leq i\leq m$, because then
$(p_i(\A))_{uv}=(\A_i)_{uv}=0$. So, we only need to check that the
entries $(p_{m+1}(\A))_{uv}, (p_{m+2}(\A))_{uv}, \ldots,
(p_{d}(\A))_{uv}$ are zero. To do this, we will show by induction
that $(p_{m+i}(A))_{yz}=0$ if $\partial(y,z)<m-i$ and
$i=0,\dots,d-m$. For $i=0$ this is clear. For $i=1$, this follows
from the equation $\A\A_m = \beta_{m-1}\A_{m-1}+
\alpha_{m}\A_{m}+\gamma_{m+1}p_{m+1}(\A)$ and a similar argument as
in case $(b)$. The induction step then follows similarly: if
$\partial(y,z)<m-i-1$, then the equation
$$\gamma_{m+i+1}p_{m+i+1}(\A)=\A p_{m+i}(\A) -
\alpha_{m+i}p_{m+i}(\A)-\beta_{m+i-1} p_{m+i-1}(\A)$$ and induction
show that $(p_{m+i+1}(\A))_{yz}=0$.

Thus our claim is proven, and by taking the entry $uv$ in the
equation
$$
p_{m-s}(\A)=\J-\sum_{i\neq m-s}p_{i}(\A),
$$
we have $(p_{m-s}(\A))_{uv}=1$.
\end{itemize}
Joining $(a),(b)$, and $(c)$, we obtain that $p_{m-s}(\A)=\A_{m-s}$.
\end{proof}

\begin{propo}
\label{prop pld curta} Let $\lceil d/2\rceil\leq m\leq D$. Then $\G$
is $m$-partially distance-regular if and only if $\G$ is
$h$-punctually distance-regular for $h=2m-d,\dots,m$.
\end{propo}
\begin{proof} This follows from applying Lemma \ref{lem curt a
llarg} repeatedly for $s=d-m+1,\dots,m$.
\end{proof}

As mentioned, this is a generalization of the following, which
follows by taking $m=D=d$.
\begin{coro}\label{D-punc->drg} \cite{fgy1b} Let $\G$ be a graph with
spectrally maximum diameter $D=d$. Then $\G$ is distance-regular
if and only if it is $D$-punctually distance-regular.
\end{coro}

The following is a new variation on this theme. Note that we
will return to the case $D=d$ in Section \ref{specmaxdiam}.
\begin{coro}
\label{prop dos en A dr} Let $\G$ be a graph with spectrally
maximum diameter $D=d$. Then $\G$ is distance-regular if and
only if it is $(D-1)$-punctually distance-regular and
$(D-2)$-punctually distance-regular.
\end{coro}

\subsection{Punctually walk-regular and punctually spectrum-regular
graphs}\label{sec:3.3} In a manner similar to the previous
sections, we will now generalize the concept of
walk-regularity. We say that a graph $\G$ is
$h$-\emph{punctually walk-regular}, for some $h\le D$, if for
every $\ell \ge 0$ the number of walks of length $\ell$ between
a pair of vertices $u,v$ at distance $h$ does not depend on
$u,v$. If this is the case, we write
$a_{uv}^{(\ell)}=(\A^{\ell})_{uv}=a_h^{(\ell)}$.

Similarly, we say that a graph $\G$ is \emph{$h$-punctually
spectrum-regular} for a given $h\leq D$ if, for any $i \le d$, the
crossed $uv$-local multiplicities of $\lambda_i$ are the same for
all vertices $u,v$ at distance $h$. In this case, we write
$m_{uv}(\lambda_i)=m_{hi}$. Notice that, for $h=0$, these concepts
are equivalent, respectively, to walk-regularity and
spectrum-regularity. As we saw, the latter two are also equivalent
to each other. In fact, as an immediate consequence of
(\ref{crossed-mul->num-walks}) and (\ref{num-walks->crossed-mul}),
the analogous result holds for any given value of $h$.
\begin{lema}
\label{pwr=psr} Let $h\le D$. Then $\G$ is $h$-punctually
walk-regular if and only if it is $h$-punctually
spectrum-regular.
\end{lema}

The following lemma turns out to be very useful for checking
punctual walk-regularity; we will use this in the proofs of
Propositions \ref{mpdr->2m+1-d} and \ref{G is d-psr}.
\begin{lema}
\label{cheking m-wr} Let $h\le D$. If, for each $\ell \le d-1$, the number of walks in
$\G$ of length $\ell$ between vertices $u$ and $v$ such that $\dist(u,v)=h$ does not depend on $u$ and $v$,
then $\G$ is $h$-punctually walk-regular. Also, if $\G$ is
bipartite and, for each $\ell \le d-2$, the number of walks in
$\G$ of length $\ell$ between vertices $u$ and $v$ such that $\dist(u,v)=h$ does not depend on $u$ and $v$,
then $\G$ is $h$-punctually walk-regular.
\end{lema}
\begin{proof}
By using the Hoffman polynomial $H$ we know that
\begin{equation}\label{cH(A)=J}
\frac{\pi_0}{n}H(\A)=\A^d+\eta_{d-1}\A^{d-1}+\cdots
+\eta_0\I=\frac{\pi_0}{n}\J.
\end{equation}
Let $u,v$ be vertices at distance $h$. Then the existence of
the constants $a_{h}^{(\ell)}$, $\ell \le d-1$, assures that
$$
a_{uv}^{(d)}=(\A^d)_{uv}=\frac{\pi_0}{n}-\eta_{d-1}a_{h}^{(d-1)}-\cdots
-\eta_0 a_{h}^{(0)}
$$
is also constant. From the fact that $\{\I,\A,\ldots,\A^d\}$ is
a basis of ${\cal A}$, it then follows that $\G$ is
$h$-punctually distance-regular. Now let $\G$ be bipartite. If
$h$ and $d$ have the same parity, then $a_h^{(d-1)}=0$, and the
result follows as in the general case. If $h$ and $d$ have
different parities, then $a_h^{(d)}=0$. Now it follows from
(\ref{cH(A)=J}) that if $a_{uv}^{(\ell)}$ is a constant for
$\ell\le d-2$, then $a_{uv}^{(d-1)}$ also is. Here we use that
$\eta_{d-1}=\delta \neq 0$ because $\G$ is bipartite (and hence
$\lambda_i=-\lambda_{d-i}$, $0\le i\le d$). Hence $\G$ is
$h$-punctually distance-regular.
\end{proof}

Next we will show that $1$-punctual walk-regularity implies
walk-regularity. Later we will generalize this result in
Proposition \ref{h-punc->walkregular}.
\begin{propo}\label{1-puncw->walk} Let $\G$ be $1$-punctually
walk-regular. Then $\G$ is walk-regular (and spectrum-regular)
with $a_0^{(\ell)}=\delta a_1^{(\ell-1)}$ for $\ell>1$, and
$m_{1i}=\frac{\lambda_i}{\lambda_0}\frac{m_i}n$ for
$i=0,1,\dots,d$.
\end{propo}
\begin{proof} For a vertex $u$ and $\ell>0$ we have that
$a_{uu}^{(\ell)}=(\A^{\ell})_{uu}=\sum_{v \in
\Gamma_1(u)}(\A^{\ell-1})_{uv} =\delta a_1^{(\ell-1)}$, which
shows that $\G$ is walk-regular with $a_0^{(\ell)}=\delta
a_1^{(\ell-1)}$. Then $\G$ is also $1$-punctually
spectrum-regular and spectrum-regular by Lemma \ref{lema3.1},
and then $\lambda_0 m_{1i} = \sum_{v \in
    \Gamma_1(u)}(\E_i)_{vu} = (\A\E_i)_{uu} = \lambda_i(\E_i)_{uu} = \lambda_i
    \frac{m_i}{n}$, which finishes the proof.
    \end{proof}

Interesting examples of punctually walk-regular graphs can be
obtained from association schemes.
\begin{propo}
\label{theo-squemes} Let $\{\B_0=\I, \B_1, \dots, \B_e\}$ be an
association scheme and let $\G$ be one of the graphs in this
scheme. If also its distance-$h$ graph $\G_h$ is in the scheme,
then $\G$ is $h$-punctually walk-regular.
\end{propo}
\begin{proof} By the assumption there are $i,k$ such that $\A=\B_i$ and
$\A_h=\B_k$.
Let $u,v$ be vertices at distance $h$ in $\G$. Because the
Bose-Mesner algebra ${\cal B}$ is closed under the ordinary product,
there are constants $c_{j\ell}$ such that
$$(\A^{\ell})_{uv}=(\B_{i}^{\ell})_{uv} = (\sum_{j=0}^e
c_{j\ell} \B_j)_{uv}=c_{k\ell}.$$ So $\G$ is $h$-punctually
walk-regular.
\end{proof}

In fact, this proposition shows that any graph in an
association scheme is $h$-punctually walk-regular for $h=0$
($\A_0=\B_0$) and $h=1$ ($\A_1=\B_i$). Note that because of our
restriction in this paper to connected graphs, we should
(formally speaking) say that each of the connected components
of a graph in an association scheme is $h$-punctually
walk-regular for $h=0,1$. Specific examples with other $h$ will
show up in the next section. Related to this observation about
graphs in association schemes is the concept of a coherent graph,
as discussed by Klin, Muzychuk, and Ziv-Av \cite{kmm}. Roughly speaking,
an (undirected connected) graph $\G$ is coherent if it is in
the smallest association scheme (coherent configuration) whose Bose-Mesner algebra contains
the adjacency algebra of $\G$.

\subsection{$m$-Walk-regular graphs}\label{sec:3.4}
In~\cite{DaFiGa09}, the concept of $m$-walk-regularity was
introduced: For a given integer $m\le D$, we say that $\G$ is
{\em $m$-walk-regular} if the number of walks $a_{uv}^{(\ell)}$
of length ${\ell}$ between vertices $u$ and $v$ only depends on
their distance $h$, provided that $h \le m$. In other words,
$\G$ is $m$-walk-regular if it is $h$-punctually walk-regular
for every $h\leq m$. Obviously, $0$-walk-regularity is the same
concept as walk-regularity.

Similarly, a graph is called $m$-spectrum-regular graph if it
is $h$-punctually spectrum-regular for all $h \le m$. By
Lemma~\ref{pwr=psr}, this is equivalent to $m$-walk-regularity.
Moreover, in~\cite{DaFiGa09}, $m$-walk-regular graphs were
characterized as those graphs for which $\A_i$ looks the same
as $p_i(\A)$ for every $i$ when looking through the `window'
defined by the matrix $\A_0+\A_1+\cdots+\A_m$. A generalization
of this will be proved in the next section.
\begin{propo} \cite{DaFiGa09}
\label{m-cami-reg=m-espect-reg} Let $m\le D$. Then $\G$ is
$m$-walk-regular (and $m$-spectrum-regular) if and only if
$p_i(\A) \circ \A_j = \delta_{ij} \A_i$ for $i=0,1,\dots,d$ and
$j =0,1,\dots,m$.
\end{propo}

This result implies the following connection with partial
distance-regularity.
\begin{propo}\label{wr->pdr} Let $m \le D$ and let $\G$ be
$m$-walk-regular. Then $\G$ is $m$-partially distance-regular
and $a_m$ (and hence $b_m$) is well-defined.
\end{propo}
\begin{proof} Proposition \ref{m-cami-reg=m-espect-reg} implies that
$\A_i=p_i(\A)$ for $i \le m$, and hence that $\G$ is
$m$-partially distance-regular, and that $\p_{m+1}(\A) \circ
\A_m = \O$. It follows that $$(\A\A_m)\circ \A_m=(\A p_m(\A))
\circ \A_m
=(\beta_{m-1}\A_{m-1}+\alpha_m\A_m+\gamma_{m+1}p_{m+1}(\A))\circ
\A_m=\alpha_m\A_m,$$ which shows that $a_m=\alpha_m$ is well-defined,
and hence also $b_m$ is well-defined.
\end{proof}

It turns out though that much weaker conditions on the number
of walks are sufficient to show $m$-partial
distance-regularity.
\begin{propo}\label{weakwalks->mpdr} Let $m \le D$.
If the number of walks in $\G$ of length $\ell$ between vertices
$u$ and $v$ depends only on $\dist(u,v)=h$ for each $h<m$,
$\ell=h,h+1$, and $h=\ell=m$, then $\G$ is $m$-partially
distance-regular.
\end{propo}
\begin{proof} If $\dist(u,v)=h \le m$, then
$a^{(h)}_h=|\G_1(u)\cap\G_{h-1}(v)|a^{(h-1)}_{h-1}$ assures
that $c_h$ is well-defined. If $\dist(u,v)=h < m$, then
similarly $a^{(h+1)}_h=|\G_1(u)\cap\G_{h}(v)|a^{(h)}_{h}+c_h
a^{(h)}_{h-1}$ assures that $a_h$ is well-defined.
\end{proof}

In the next section, we shall further work out the difference
between $m$-partial distance-regularity and
$m$-walk-regularity. The following characterization by
Rowlinson \cite{r97} (see also Fiol~\cite{fiol02}) follows
immediately from Proposition \ref{m-cami-reg=m-espect-reg}.
\begin{propo} \cite{r97}
A graph is $D$-walk-regular if and only if it is
distance-regular.
\end{propo}

In the previous section we showed that any graph $\G$ in an
association scheme is $1$-walk-regular. In case the
distance-matrices $\A_h$ of $\G$ are in the association scheme
for all $h \le m$, then the graph is clearly $m$-walk-regular
by Proposition \ref{theo-squemes}. Such graphs are examples of
so-called distance($m$)-regular graphs, as introduced by Powers
\cite{p91}. A graph is called {\em distance($m$)-regular} if
for every vertex $u$ there is an equitable partition $\{\{u\},
\Gamma_1(u),\dots,\Gamma_m(u),V_{m+1}(u),\dots,V_{e}(u)\}$ of
the vertices, with quotient matrix being the same for every $u$
(we refer the reader who is unfamiliar with equitable
partitions to \cite[p. 79]{g93}). We observe that this is
equivalent to the existence of $(0,1)$-matrices
$\B_{m+1},\dots, \B_e$ that add up to $\A_{m+1}+\cdots
+\A_{D}$, such that the linear span of the set
$\{\A_0,\A_1,\dots,\A_m,\B_{m+1},\dots, \B_e\}$ is closed under
left multiplication by $\A$. Consequently, a
distance($m$)-regular graph is $m$-walk-regular (the same
argument as in the proof of Proposition \ref{theo-squemes}
applies). We now present some interesting examples of
distance($m$)-regular graphs (mostly coming from association
schemes).

The bipartite incidence graph of a square divisible design with
the dual property (i.e., such that the dual design is also
divisible with the same parameters as the design itself) is a
distance(2)-regular graph with $D=4$ (and in general $d=5$).
This follows for example from the distance distribution diagram
(see \cite[p. 24]{bcn}); hence these graphs are
$2$-walk-regular.

The distance-4 graph of the distance-regular Livingstone graph
is a distance(2)-regular graph with $D=3$ (and $d=4$); again,
see the distribution diagram \cite[p. 407]{bcn}.

The graph defined on the 55 flags of the symmetric
$2$-$(11,5,2)$ design, with flags $(p,b)$ and $(p',b')$ being
adjacent if also $(p,b')$ and $(p',b)$ are flags is
distance(3)-regular with $D=4$ and $d=5$; see the distribution
diagram in Figure \ref{55flags}.

\begin{figure}[h!]
\begin{center}
\resizebox{80mm}{!}{\includegraphics{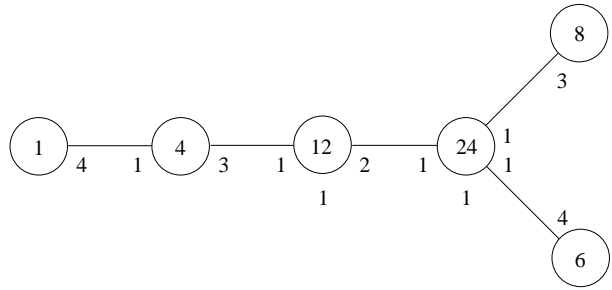}} \caption{Distance distribution diagram
of the flag graph} \label{55flags}
\end{center}
\end{figure}

The above examples show that there are $(D-1)$-walk-regular
graphs with diameter $D$ that are not distance-regular, for
small $D$. For larger $D$, we do not have such examples
however, so the question arises if these exist at all.
\begin{problem}\label{problem2} (a) Determine the smallest $m=m_{wr,D}(D)$ such that
every $m$-walk-regular graph with diameter $D$ is distance-regular.

(b) Determine the smallest $m=m_{wr,d}(d)$ such that every
$m$-walk-regular graph with $d+1$ distinct eigenvalues is
distance-regular.
\end{problem}

Note that a $(d-1)$-walk-regular graph (with $d-1 \le D$) is
distance-regular by Propositions \ref{wr->pdr} and \ref{D/d-1
suffices}.

Another interesting example related to this problem is the
graph F234B from the Foster Census \cite{rcmd09}. This graph
has $D=8$, $d=11$, it is 5-arc-transitive, and hence
5-walk-regular. The vertices correspond to the 234 triangles in
$PG(2,3)$ with two vertices being adjacent whenever the
corresponding triangles have one common point and their
remaining four points are distinct and collinear \cite[p.
125]{biggs}. This and the above examples suggest that
$m_{wr,D}(D) > \frac D2+1$.

\subsection{$(\ell,m)$-Walk-regular graphs}\label{sec:lmwr}

In order to understand the difference between $m$-partial
distance-regularity and $m$-walk-regularity, the following
generalization of the latter is useful. As before, let $\G$ be a
graph with diameter $D$ and $d+1$ eigenvalues. Given two
integers $\ell\le d$ and $m\le D$ satisfying $\ell\ge m$, we
say that is $\G$ is {\em $\ell$-partially $m$-walk-regular}, or
\emph{$(\ell,m)$-walk-regular} for short, if the number of
walks of length $\ell'\le \ell$ between any pair of vertices
$u,v$ at distance $m'\le m$ does not depend on such vertices
but depends only on $\ell'$ and $m'$. The concept of
$(\ell,m)$-walk-regularity was introduced in \cite{DaFiGa09b},
and generalizes some of the concepts from the previous
sections. In fact, the following equivalences follow
immediately:
\begin{itemize}
\item $~(d,0)$-walk-regular graph \quad $\equiv$ \quad
    walk-regular graph
\item $(d,m)$-walk-regular graph \quad $\equiv$ \quad
    $m$-walk-regular graph
\item $(d,D)$-walk-regular graph \quad $\equiv$ \quad
    distance-regular graph
\end{itemize}

We also note that $(\ell,0)$-walk-regular graphs were
introduced in \cite{fg1} under the name of $\ell$-partially
walk-regular graphs, and they were also studied by Huang et al.
\cite{hhlw07}. More relations can be derived from the following
generalization of Proposition \ref{m-cami-reg=m-espect-reg}.
Here we will give a new (and shorter) proof.
\begin{propo} \cite{DaFiGa09b} \label{teor (l,m)-equivalencies}
Let $d \ge \ell \ge m \le D$. Then $\G$ is
$(\ell,m)$-walk-regular if and only if $p_i(\A)\circ\A_j=
\delta_{ij}\A_i$ for $i=0,1,\dots,\ell$ and $j=0,1,\dots,m$.
\end{propo}
\begin{proof}
Assume the latter. Let $x^h=\sum_{i=0}^h\eta_{ih}p_i$ for
$h\leq \ell$. Then for each pair of vertices $u,v$ at distance
$j\leq m$, and $h\leq \ell$, we have:
$$(\matrixA^h)_{uv}  = (\matrixA^h\circ\matrixA_j)_{uv}=
\sum_{i=0}^h\eta_{ih}\left(p_i(\matrixA)
\circ\matrixA_j\right)_{uv}=\eta_{jh}.$$ Consequently, $\G$ is
$(\ell,m)$-walk-regular. Conversely, consider the mapping
$\Phi:\R_{\ell}[x]\to\R^{m+1}$ defined by
$\Phi(p)=(\varphi_0(p),\ldots,\varphi_m(p))$, with
$p(\matrixA)\circ \A_j=\varphi_j(p)\matrixA_j$. This mapping is
linear and
$\Phi(x^j)=(\varphi_0(x^j),\ldots,\varphi_j(x^j),0,\ldots,0)$
with $\varphi_j(x^j)\neq 0$, for $j=0,1,\ldots,m$. Therefore
the restriction $\widetilde{\Phi}$ of $\Phi$ to $\R_m[x]$, is
one-to-one. Now, let
$r_i=\widetilde{\Phi}^{-1}(0,\ldots,1,\ldots,0)$, with the $1$
in the $i$-th position, for $i\le m$. In other words,
$r_i(\A)\circ \A_j= \delta_{ij} \A_i$ for $i,j \le m$. Each
polynomial $r_i$ satisfies $r_i(\matrixA)=\sum_{j=0}^m
r_i(\matrixA)\circ \A_j=\matrixA_i$, and therefore $r_i=p_i$ by
Lemma \ref{lema3.1}. Thus, $p_i(\A)\circ \A_j= \delta_{ij}
\A_i$ for $i,j \le m$.

Now let $m+1\leq i\leq\ell$ and $j\leq m$. Then
$p_i(\matrixA)\circ p_j(\matrixA)= p_i(\matrixA)\circ
\matrixA_j=\varphi_j(p_i)\matrixA_j$. From this equation, we
find that $\varphi_j(p_i)p_j(\lambda_0) =\varphi_j(p_i)\frac1n
\som(\A_j)=\frac1n \som(p_i(\A) \circ p_j(\A))=\langle p_i,p_j
\rangle=0.$ Thus, $\varphi_j(p_i)=0$ and $p_i(\matrixA)\circ
\matrixA_j=\matrix0$, which completes the proof.
\end{proof}

The following equivalences now follow; see also the proof of
Proposition \ref{wr->pdr}.
\begin{itemize}
\item $~~~~~(m,m)$-walk-regular graph \quad $\equiv$ \quad
    $m$-partially distance-regular graph
\item $(m+1,m)$-walk-regular graph \quad $\equiv$ \quad
    $m$-partially distance-regular graph \\ \mbox{\hskip
    6.5cm} with $a_m$ (and hence $b_m$) well-defined
\end{itemize}

We have seen in Proposition \ref{weakwalks->mpdr} though that
weaker conditions on the number of walks are sufficient to show
$m$-partial distance-regularity.

An example illustrating the above is the unique $(6,5)$-cage on 40 vertices obtained from the
Hoffman-Singleton graph by removing an induced Petersen graph.
This generalized Moore graph has $d=4$, $D=3$, and girth 5.
From its distance distribution diagram (see \cite[Fig. 9.1]{kmm}), it follows that
it is 2-partially distance-regular, but not $(3,2)$-walk-regular.

The next proposition follows from the characterization in
Proposition \ref{teor (l,m)-equivalencies}. It clarifies the role of
the preintersection numbers given by the expressions in
$(\ref{preintersec})$.
\begin{propo} \cite{DaFiGa09b}
\label{propointersec} Let $d \ge \ell \ge m \le D$, let $\G$ be
$(\ell,m)$-walk-regular, and let $i,j,k\le m$. If $i+j\le
\ell$, then the preintersection  number  $\xi_{ij}^k$ equals
the well-defined intersection number $p_{ij}^k$. If $i+j\ge
\ell+1$, then the preintersection number $\xi_{ij}^k$ equals
the average $\overline{p}_{ij}^k$ of the values
$p_{ij}^k(u,v)=|\G_i(u)\cap\G_j(v)|$ over all vertices $u,v$ at
distance $k$.
\end{propo}

The graph F084A from the Foster Census \cite{rcmd09} has $D=7$
and $d=10$. It is $2$-walk-regular, $3$-partially
distance-regular, and all intersection numbers $c_i,
i=1,2,\dots,7$ are well-defined. This implies that the number
of walks of length $\ell$ between vertices at distance $\ell$
depends only on $\ell$. Still, this graph is not even
$(4,3)$-walk-regular, because $a_3$ is not well-defined.

We will now obtain relations between various kinds of partial
walk-regularity.
\begin{propo}
\label{(l,m)->(l+1,m-1)} Let $d-1 \ge \ell \ge m \ge 1$, $m \le
D$, and let $\G$ be $(\ell,m)$-walk-regular. Then $\G$ is
$(\ell+1,m-1)$-walk-regular.
\end{propo}
\begin{proof} Let $u,v$ be two vertices of $\G$ at distance $j\leq m-1$,
with $j<\ell-1$ (if $m=\ell$). From
$\gamma_{\ell+1}p_{\ell+1}=xp_{\ell}-\beta_{\ell-1}p_{\ell-1}-
\alpha_{\ell}p_{\ell}$ we have:
\begin{eqnarray*}
\gamma_{\ell+1}(p_{\ell+1}(\A)\circ\A_j)_{uv} & = & (\A
p_{\ell}(\A)\circ\A_j)_{uv}=
(\A p_{\ell}(\A))_{uv}=\\
& & \sum_w\A_{uw}(p_{\ell}(\A))_{wv}=
\sum_{\partial(w,u)=1}(p_{\ell}(\A))_{wv}=0\,,
\end{eqnarray*}
since $\partial(w,v)\le j+1\leq m$, $\partial(w,v)< \ell$ if
$m=\ell$, and $p_{\ell}(\A)\circ\A_i=\matrix0$ for $i\leq
m<\ell$. Moreover, if $m=\ell$ and $j=\ell-1$ then $\G$ is
$\ell$-partially distance-regular. Thus, we get
$$
\gamma_{\ell+1}(p_{\ell+1}(\A)\circ\A_{\ell-1})_{uv} =(\A
\A_{\ell})_{uv}-b_{\ell-1}(\A_{\ell-1})_{uv}=0,
$$
since $p_i(\A)=\A_i$, $0\le i\le \ell$, and
$b_{\ell-1}=\beta_{\ell-1}=(\A \A_{\ell})_{uv}$ is well-defined.
Therefore, $p_{\ell+1}(\A)\circ\A_j=\matrix0$ for every $j\leq m-1$,
and Proposition \ref{teor (l,m)-equivalencies} yields the result.

Alternatively, notice that, if $\G$ is $(\ell,m)$-walk-regular,
then the number of walks of length $\ell+1$ between vertices
$u,v$ at distance $j<m$ equals
$$
a_{uv}^{(\ell+1)}=c_j a_{j-1}^{(\ell)}+a_j a_j^{(\ell)}+b_j
a_{j+1}^{(\ell)}
$$ and hence is a constant $a_j^{(\ell+1)}$.
\end{proof}

As a direct consequence of this last result, we have that
$(\ell,m)$-walk-regularity implies $(\ell+r,m-r)$-walk-regularity
for every integer $r \le d-\ell$ and $1\le r\le m$. In particular,
every $(\ell,m)$-walk-regular graph with $\ell\ge d-m$ is also
walk-regular. Also the following connections between partial
distance-regularity and $m$-walk-regularity follow.

\begin{propo}\label{mpdr->2m+1-d} Let $m \le D$ and let $\G$ be $m$-partially
distance-regular. If $m \ge \frac{d-1}2$, then $\G$ is
$(2m+1-d)$-walk-regular. If $m \ge \frac{d-2}2$ and $a_m$ is
well-defined, then $\G$ is $(2m+2-d)$-walk-regular. If $m \ge
\frac{d-3}2$ and $\G$ is bipartite, then $\G$ is
$(2m+3-d)$-walk-regular.
\end{propo}
\begin{proof} For the first statement, observe that $\G$ is
$(m,m)$-walk-regular, so by Proposition \ref{(l,m)->(l+1,m-1)}
it is $(d-1,2m+1-d)$-walk-regular. By Lemma \ref{cheking m-wr},
$\G$ is therefore $(2m+1-d)$-walk-regular. The proof of the
second statement is similar, starting from
$(m+1,m)$-walk-regularity. Also for the third statement we can
start from $(m+1,m)$-walk-regularity, because $a_m=0$ is well-defined
for a bipartite graph. Now it follows that $\G$ is
$(d-2,2m+3-d)$-walk-regular, and by Lemma \ref{cheking m-wr},
$\G$ is $(2m+3-d)$-walk-regular.
\end{proof}

Note that this proposition also relates Problems \ref{problem1} and
\ref{problem2}. For example, if $m_{pdr}(d)=d-1$ (for some $d$),
then there is a $(d-2)$-partially distance-regular graph that is not
distance-regular. This graph would be $(d-3)$-walk-regular by the
proposition, which would imply that $m_{wr,d}(d)\ge d-2$. In general
it shows that $m_{wr,d}(d) \ge 2m_{pdr}(d)-d$.

As it is known, graphs with few distinct eigenvalues have many
regularity features. For instance, every (regular, connected)
graph with three distinct eigenvalues is strongly regular (that
is, distance-regular with diameter two). Any graph with four
distinct eigenvalues is known to be walk-regular, and the
bipartite ones with four distinct eigenvalues are always
distance-regular. This also follows from Propositions
\ref{mpdr->2m+1-d} ($d=3, m=1$) and \ref{D/d-1 suffices}. Moreover,
if $\G$ has four distinct eigenvalues and
$a_1$ is well-defined, then it is $1$-walk-regular. If in addition
$c_2$ is well-defined, then the graph is distance-regular by
Proposition \ref{D/d-1 suffices}. Similarly, if $\G$ is a
bipartite graph with five distinct eigenvalues then $\G$ is
$1$-walk-regular. Moreover, if $c_2$ is well-defined, then $\G$
is distance-regular.

A natural question would be to find out when the converse of
Proposition \ref{(l,m)->(l+1,m-1)} is true. At least the
following can be said (we omit the proofs):
\begin{propo}
Let $m \le D, m \le d-1$. Then $\G$ is $(m,m)$-walk-regular if
and only if it is $(m+1,m-1)$-walk-regular and the
      intersection number $c_m$ is well-defined.
\end{propo}
\begin{propo}
  Let $m \le D, m \le d-2$.
 Then $\G$ is $(m+1,m)$-walk-regular if and
      only if it is $(m+2,m-1)$-walk-regular and the
      intersection numbers $c_m$,
$a_m$, and $b_m$ are well-defined.
\end{propo}

It seems complicated to extend this further; for example,
$(m+2,m)$-walk-regularity implies $(m+3,m-1)$-walk-regularity,
but for the reverse we do not know how to avoid using that
$c_{m+1}$ is well-defined (besides $c_m$, $a_m$, $b_m$).
But $(m+2,m)$-walk-regularity does
not necessarily imply that $c_{m+1}$ is well-defined.

An interesting example is the graph F168F from the Foster
Census \cite{rcmd09}; it is a (bipartite) graph with $D=8$ and
$d=20$. The intersection numbers are well-defined up to $b_5$,
so the graph is $(6,5)$-walk-regular, and hence also
$(7,4)$-walk-regular. Moreover, it is $(10,3)$-walk-regular,
and $2$-walk-regular.

As a final result in this section, we generalize Proposition
\ref{1-puncw->walk}. Note that every (regular) graph is
$(\ell,0)$-walk-regular for $\ell \le 2$, and that $q_h=x$ for
$h=1$.
\begin{propo}\label{h-punc->walkregular} Let $h \le D$ and
let $\G$ be $h$-punctually distance-polynomial, with
$\A_h=q_h(\A)$. Let $\ell+1$ be the number of distinct
eigenvalues $\lambda_i$ for which $q_h(\lambda_i)=0$. If $\G$ is
$h$-punctually spectrum-regular and $(\ell,0)$-walk-regular,
then it is walk-regular (and spectrum-regular) and
\begin{equation}\label{m_{hi}}
m_{hi}=\frac{q_h(\lambda_i)}{q_h(\lambda_0)}\frac{m_i}n\qquad (i=0,1,\dots,d).
\end{equation}
\end{propo}
\begin{proof} Let $\mathcal{I}$ denote the set of indices $i$ such
that $q_h(\lambda_i)=0$, so $|\mathcal{I}|=\ell+1$. If $\G$ is
$h$-punctually spectrum-regular then
$$
q_h(\lambda_0) m_{hi}=\sum_{v \in
\Gamma_h(u)}(\E_i)_{vu}=(\A_h\E_i)_{uu}=(q_h(\A)\E_i)_{uu} =
q_h(\lambda_i)(\E_i)_{uu}\qquad (u\in V),
$$
which shows that $m_u(\lambda_i)=(\E_i)_{uu}$ is a constant,
and $m_{0i}=\frac{q_h(\lambda_0)}{q_h(\lambda_i)}m_{hi}$, for
every $i\not\in \mathcal{I}$. Moreover, if $\G$ is
$(\ell,0)$-walk-regular, then (\ref{crossed-mul->num-walks})
yields:
$$
\sum_{i\in \mathcal{I}} m_u(\lambda_i) \lambda_i^{\ell'} =
a^{(\ell')}-\sum_{i\not\in \mathcal{I}} m_{0i} \lambda_i^{\ell'}\qquad (0\le \ell'\le \ell).
$$
This is a linear system of $\ell+1$ equations with $\ell+1$
unknowns $m_u(\lambda_i)$, and this system has a unique
solution as it has a Vandermonde matrix of coefficients. Hence
$m_u(\lambda_i)=\frac{m_i}{n}$ for all $0\le i\le d$ and we get
(\ref{m_{hi}}).
\end{proof}

With reference to (\ref{m_{hi}}), we note that the
multiplicities $m_i$ can be computed from the highest degree
predistance polynomial as $
m_i=(-1)^i\frac{\pi_0p_d(\lambda_0)}{\pi_ip_d(\lambda_i)}$, cf.
\cite{fg2}.

\section{Spectral distance-degree
characterizations}\label{sec:4}

In this section we will obtain results that have the same
flavor as the spectral excess theorem \cite{fg2}. This theorem
states that the average degree $\overline{\delta}_d$ of the
distance-$d$ graph is at most $p_d(\lambda_0)$ with equality if
and only if the graph is distance-regular (for short proofs of
this theorem, see \cite{vd08,fgg09}). The following result
gives a quasi-spectral characterization of punctually
distance-polynomial graphs, in terms of the average degree
$\overline{\delta}_h=\frac1n \som(\A_h)$ of the distance-$h$
graph $\G_h$ and the {\em average crossed local multiplicities}
$$\overline{m}_{hi}=\frac1{n\overline{\delta}_h}\sum_{\partial(u,v)=h}m_{uv}(\lambda_i).$$
\begin{propo}
\label{propo h-punctual d-p} Let $h \le D$. Then
$$
 \overline{\delta}_h\le\frac
1n\left(\sum_{i=0}^d\frac{\overline{m}_{hi}^2}{m_i}\right)^{-1}
$$
with equality if and only if $\G$ is $h$-punctually
distance-polynomial. If $\A_h=q_h(\A)$, then
$$
\delta_h=q_h(\lambda_0) \qquad and \qquad \overline{m}_{hi}=
\frac{q_h(\lambda_i)}{q_h(\lambda_0)}\frac{m_i}{n}\qquad (i=0,1,\dots,d).
$$
\end{propo}
\begin{proof}
We denote by $\widetilde{\A_h}$ the orthogonal projection of
$\A_h$ onto ${\cal A}$. By using the orthogonal basis
consisting of the matrices $\E_i=\lambda_i^*(\A)$,
$i=0,1,\ldots,d$, we have
$$
\widetilde{\A_h}=\sum_{i=0}^d \frac{\langle \A_h,\E_i\rangle}
{\|\E_i\|^2}\E_i=\sum_{i=0}^d \frac
1{m_i}\left(\sum_{\partial(u,v)=h}(\E_i)_{uv}\right)\E_i=
n\overline{\delta}_h\sum_{i=0}^d\frac{\overline{m}_{hi}}{m_i}\E_i.
$$
Hence the orthogonal projection of $\A_h$ onto ${\cal A}$ is
the matrix $q_h(\A)$, where
\begin{equation}\label{q_h}
q_h=n\overline{\delta}_h\sum_{i=0}^d\frac{\overline{m}_{hi}}{m_i}\lambda_i^*.
\end{equation}
Since
$$
\|\widetilde{\A_h}\|^2=\langle q_h,q_h\rangle=
n^2\overline{\delta}_h^2\sum_{i=0}^d\frac{\overline{m}_{hi}^2}{m_i^2}\,\frac{m_i}{n}
=n\overline{\delta}_h^2\sum_{i=0}^d\frac{\overline{m}_{hi}^2}{m_i}
$$
and $\|\A_h\|^2=\overline{\delta}_h$, the upper bound on
$\overline{\delta}_h$ follows from $\|\widetilde{\A_h}\|\le
\|\A_h\|$. Moreover, Pythagoras's theorem says that the scalar
condition $\|\widetilde{\A_h}\|= \|\A_h\|$ is equivalent to
$\A_h\in{\cal A}$ and hence to $\G$ being $h$-punctually
distance-polynomial. Moreover, it shows that if $\G$ is
punctually distance-polynomial, then $\A_h=q_h(\A)$, with $q_h$
as given in (\ref{q_h}). It follows from Lemma \ref{lema3.1}
that $\G_h$ is regular of degree
$\overline{\delta}_h=\delta_h=q_h(\lambda_0)$. Moreover, from
(\ref{q_h}) it follows that
$q_h(\lambda_i)=n\overline{\delta}_h\frac{\overline{m}_{hi}}{m_i}$,
and this gives the required expression for $\overline{m}_{hi}$.
\end{proof}

Let $\overline{a}_h^{(\ell)}$ be the average number of walks of
length $\ell$ between vertices at distance $h\le D$, and recall
from (\ref{omega_k}) that the leading coefficient $\omega_h$ of
$p_h$ satisfies $\omega_h^{-1}=\gamma_1\gamma_2\cdots
\gamma_h$. Now the following results are variations of
Proposition \ref{propo h-punctual d-p} for punctual
distance-regularity.
\begin{propo}
\label{pdr2cond} Let $h \le D$. Then
$$
 \overline{\delta}_h\le\frac
{p_h(\lambda_0)}{[\omega_h \overline{a}_h^{(h)}]^2}
$$
with equality if and only if $\G$ is $h$-punctually
distance-regular, which is the case if and only if
$\overline{a}_h^{(h)}= \gamma_1\gamma_2\cdots \gamma_h$ and
$\overline{\delta}_h=p_h(\lambda_0)$.
\end{propo}
\begin{proof}
First, observe that
$$
\langle \A_h,p_h(\A)\rangle  =  \frac 1n
\sum_{\partial(u,v)=h}(p_h(\A))_{uv}=
\frac{\omega_h}{n}
\sum_{\partial(u,v)=h}a^{(h)}_{uv}=\omega_h\overline{\delta}_h \overline{a}_h^{(h)}.
$$
Thus, the orthogonal projection  of $\A_h$ onto $\langle
p_h(\A) \rangle$ is
$\breve{\A_h}=\frac{\omega_h\overline{\delta}_h
\overline{a}_h^{(h)}}{p_h(\lambda_0)}\p_h(\A)$,  and
$$
\frac{[\omega_h\overline{\delta}_h \overline{a}_h^{(h)}]^2}{p_h(\lambda_0)}=
\|\breve{\A_h}\|^2 \le  \|\A_h\|^2=\overline{\delta}_h
$$
gives the claimed inequality for $\overline{\delta}_h$
(alternatively, it follows from Cauchy-Schwarz). As before, it
is clear that equality holds if and only if
$\A_h=\breve{\A_h}$. Using Lemma \ref{lema3.1}, this is
equivalent to $\A_h=p_h(\A)$ ($\G$ being $h$-punctually
distance-regular). Equality thus implies that
$\overline{\delta}_h=p_h(\lambda_0)$ and hence that
$\overline{a}_h^{(h)}=\omega_h^{-1}= \gamma_1\gamma_2\cdots
\gamma_h$. To complete the argument, note that the latter
implies that equality holds in the inequality.
\end{proof}

The bound of Proposition \ref{propo h-punctual d-p} is more
restrictive than that of Proposition \ref{pdr2cond}. This
follows from the fact that $\A_h$ and $\widetilde{\A_h}$ have
the same projection $\breve{\A_h}$ onto $\langle p_h(\A)
\rangle$, and hence that $\|\breve{\A_h}\|\le
\|\widetilde{\A_h}\| \le \|\A_h\|$. This means that the bound
of Proposition \ref{propo h-punctual d-p} is sandwiched between
the average degree of $\G_h$ and the bound of Proposition
\ref{pdr2cond}. Thus, the tighter the latter bound is, the
tighter the first one is. For a better comparison of the
bounds, notice that a simple computation gives that
$$
\overline{a}_h^{(h)}=\sum_{i=0}^d \overline{m}_{hi}\lambda_i^h =
\frac1{\omega_h} \sum_{i=0}^d\overline{m}_{hi}p_h(\lambda_i)  \qquad (i=0,1,\dots,d).
$$
We thus find that
$$
\overline{\delta}_h\le
\frac1n\left(\sum_{i=0}^d\frac{\overline{m}_{hi}^2}{m_i}\right)^{-1} \le \frac
{p_h(\lambda_0)}{\omega_h^2}\left(\sum_{i=0}^d \overline{m}_{hi}\lambda_i^h\right)^{-2}
=p_h(\lambda_0)\left(\sum_{i=0}^d\overline{m}_{hi}p_h(\lambda_i)\right)^{-2}.$$

As we shall see in more detail in the next section, Proposition
\ref{pdr2cond} is a generalization of the spectral excess
theorem, at least if we combine it with Corollary
\ref{D-punc->drg}. For the next proposition this is also the
case; by considering the case $h=D=d$.
\begin{propo}
\label{prop escalar1 hdr } Let $h \le D$ and let $\G$ be such
that $\langle p_i(\A),\A_h\rangle=0$ for $i=h+1,\dots,d$. Then
$\overline{\delta}_h\leq p_h(\lambda_0)$ with equality if and
only if $\G$ is $h$-punctually distance-regular.
\end{propo}
\begin{proof}
The orthogonal projection of $\A_h$ onto ${\cal A}$ is
\begin{eqnarray*}
\widetilde{\A_h} & = &
\sum_{i=0}^d\frac{\langle\A_h,p_i(\A)\rangle}
{\|p_i(\A)\|^2}p_i(\A)= \frac{\langle\A_h,p_h(\A)\rangle}
{\|p_h(\A)\|^2}p_h(\A)= \frac{\langle\A_h,H(\A)\rangle}
{\|p_h(\A)\|^2}p_h(\A) \\
 & = & \frac{\langle\A_h,\J\rangle}
{p_h(\lambda_0)}p_h(\A)
 =  \frac{\langle \A_h,\A_h\rangle}
{p_h(\lambda_0)}p_h(\A) = \frac{\overline{\delta}_h}{p_h(\lambda_0)}p_h(\A).
\end{eqnarray*}
We have $\|\A_h\|^2=\overline{\delta}_h$ and $\displaystyle
\|\widetilde{\A_h}\|^2=
\frac{\overline{\delta}_h^2}{p_h(\lambda_0)}$. From
$\|\widetilde{\A_h}\|\leq\|\A_h\|$, we obtain
$\overline{\delta}_h\leq p_h(\lambda_0)$. From Pythagoras's
theorem, equality gives $\A_h=\widetilde{\A_h}=p_h(\A)$.
\end{proof}

By projection onto ${\cal D}$ we obtain the following `dual'
result.
\begin{propo} \label{prop escalar2 hdr }
Let $h \le D$ and let $\G$ be such that $ \langle
p_h(\A),\A_i\rangle=0$ for $i=0,\dots,h-1$. Then
$\overline{\delta}_h\geq p_h(\lambda_0)$ with equality if and
only if $\G$ is $h$-punctually distance-regular.
\end{propo}
\begin{proof}
We now consider the orthogonal projection $\widehat{p_h(\A)}$
of $p_h(\A)$ onto ${\cal D}$:
\begin{eqnarray*}
\widehat{p_h(\A)} & = & \sum_{i=0}^D\frac{\langle
p_h(\A),\A_i\rangle} {\|\A_i\|^2}\A_i= \sum_{i=0}^{h}\frac{\langle
p_h(\A),\A_i\rangle} {\|\A_i\|^2}\A_i= \frac{\langle
p_h(\A),\A_h\rangle} {\|\A_h\|^2}\A_h \\
 & = & \frac{\langle
p_h(\A),\J\rangle} {\overline{\delta}_h}\A_h
 =  \frac{\langle p_h(\A),p_h(\A)\rangle}
{\overline{\delta}_h}\A_h = \frac{p_h(\lambda_0)}{\overline{\delta}_h}\A_h.
\end{eqnarray*}
From this we now obtain that
$\frac{(p_h(\lambda_0))^2}{\overline{\delta}_h}=
\|\widehat{p_h(\A)}\|^2\leq\|p_h(\A)\|^2 =p_h(\lambda_0)$, and
hence that $\overline{\delta}_h\geq p_h(\lambda_0)$. Moreover,
equality gives $\A_h=\widehat{p_h(\A)}= p_h(\A)$.
\end{proof}

From the latter two propositions, we obtain the following
result.
\begin{coro}
\label{prop que maco 12} Let $h \le D$. Then $\G$ is
$h$-punctually distance-regular if and only if \linebreak$
\langle p_h(\A),\A_i\rangle=0$ for $i=0,\dots,h-1$ and $\langle
p_i(\A),\A_h\rangle=0$ for $i=h+1,\dots,d$.
\end{coro}

\section{Graphs with spectrally maximum diameter}\label{specmaxdiam}

In this section we focus on the important case of graphs with
spectrally maximum diameter $D=d$. Distance-regular graphs are
examples of such graphs. In this context, we first recall the
following characterizations of distance-regularity. We include
a new proof for completeness.
\begin{propo}
[Folklore] The following statements are equivalent:
\begin{itemize}
\item[{\em(i)}] $\G$ is distance-regular,
\item[{\em(ii)}] ${\cal D}$ is an algebra with the ordinary
    product,
\item[{\em(iii)}] ${\cal A}$ is an algebra with the
    Hadamard product,
\item[{\em(iv)}] ${\cal A}={\cal D}$.
\end{itemize}
\end{propo}
\begin{proof}
We already observed in Section \ref{subsec_alg} that (i) and
(iv) are equivalent, and that these imply (ii) and (iii). So we
only need to prove that both (ii) and (iii) imply (iv).\\
  (ii) $\Rightarrow$ (iv): As $\A=\A_1\in{\cal D}$,
we have that $\A^k\in{\cal D}$ for any $k \ge 0$. Thus, ${\cal
A}\subset{\cal D}$ and, since $\dim {\cal A}=d+1\geq
D+1=\dim {\cal D}$, we get ${\cal A}={\cal D}$.\\
  (iii) $\Rightarrow$ (iv): As $\E_i \circ \A^j \in {\cal
  A}$, we have that $\E_i \circ \A^j =q_{ji}(\A)$ for some polynomial $q_{ji}$,
  and this polynomial clearly has degree at most $j$.
  Let $\psi_{ji}$ be the coefficient of $x^j$ in $q_{ji}$, then
  it follows that
  $(\E_i)_{uv}(\A^j)_{uv}=\psi_{ji}(\A^j)_{uv}$ for vertices $u,v$
  at distance $j$, and hence that $(\E_i)_{uv}=\psi_{ji}$. It
  thus follows that $\E_i=\sum_{j} \psi_{ji}\A_j \in {\cal D}$.
  Therefore ${\cal A}\subset{\cal D}$ and, as before, we obtain ${\cal A}={\cal D}$.
\end{proof}

\subsection{Partially distance-regular graphs}

We already observed in Section \ref{sec:pdr} that if a graph
with $D=d$ is $h$-punctually distance-polynomial, then it is
$h$-punctually distance-regular. The following, which is a bit
stronger, is an immediate consequence of Lemmas \ref{lema3.1}
and \ref{lem D to A}.

\begin{coro}
\label{lem D simetria A} Let $h\le D$ and let $\G$ have
spectrally maximum diameter $D=d$. Then $\A_h\in{\cal A}$ if
and only if $p_h(\A)\in{\cal D}$, in which case $\A_h=p_h(\A)$.
\end{coro}

It is also clear that if a graph with $D=d$ is $m$-partially
distance-polynomial, then it is $m$-partially distance-regular.
If we let ${\cal A}_{m} = \span \{\I, \A, \A^2, \ldots,
\A^{m}\}$ and ${\cal D}_{m} = \span
\{\I,\A,\A_2,\ldots,\A_m\}$, then we obtain the following by
extending the previous corollary.

\begin{coro}
Let $m \leq D$ and let $\G$ have spectrally maximum diameter $D=d$.
Then the following statements are equivalent: $\G$ is $m$-partially
distance-regular, ${\cal D}_{m} \subset {\cal A}$, ${\cal A}_{m}
\subset {\cal D}$, and ${\cal A}_{m} ={\cal D}_{m}.$
\end{coro}

\subsection{Punctually walk-regular graphs}

Graphs with spectrally maximum diameter turn out to be
$d$-punctually walk-regular. This will be used in the next
section to show the relation of Propositions \ref{propo
h-punctual d-p} and \ref{pdr2cond} to the spectral excess
theorem.
\begin{propo}
\label{G is d-psr} Let $\G$ have spectrally maximum diameter
$D=d$. Then it is both $d$-punctually walk-regular and
$d$-punctually spectrum-regular with parameters
$$
a_d^{(d)}=\frac{\pi_0}{n}=\gamma_1\gamma_2\cdots \gamma_d,\qquad m_{di}=(-1)^i\frac
{\pi_0}{n\pi_i}\qquad (i=0,\dots, d).
$$
If $\G$ is bipartite, then it is both $(d-1)$-punctually
walk-regular and $(d-1)$-punctually spectrum-regular with
parameters
$$
a_{d-1}^{(d-1)}=\frac{\pi_0}{n\delta}=\gamma_1\gamma_2\cdots
\gamma_{d-1},\qquad m_{d-1,i}=(-1)^i\frac
{\pi_0}{n\pi_i} \frac{\lambda_i}{\delta} \qquad (i=0,\dots, d).
$$
\end{propo}
\begin{proof}
It follows from Lemma \ref{cheking m-wr} and its proof that $\G$
is $d$-punctually walk-regular with
$a_d^{(d)}=\frac{\pi_0}{n}$. The latter equals
$\gamma_1\gamma_2\cdots \gamma_d$ by (\ref{omega_k}) and
(\ref{polHof}). Then by Lemma \ref{pwr=psr}, $\G$ is also
$d$-punctually spectrum-regular. Now observe that if $u,v$ are
vertices at distance $d$, then
$m_{di}=(\E_i)_{uv}=\lambda_i^*(\A)_{uv}=\frac{(-1)^i}{\pi_i}a_d^{(d)}=
(-1)^i\frac{\pi_0}{n\pi_i}$.

If $\G$ is bipartite, then it follows from Lemmas \ref{cheking
m-wr} and \ref{pwr=psr} that $\G$ is $(d-1)$-punctually
walk-regular and $(d-1)$-punctually spectrum-regular. Moreover,
it is clear that $a_d^{(d)}=\delta a_{d-1}^{(d-1)}$, hence
$a_{d-1}^{(d-1)}=\frac{\pi_0}{n\delta}=\gamma_1\gamma_2\cdots
\gamma_{d-1}$ (because $\gamma_d=\delta$ for a bipartite
graph). If $u,v$ are vertices at distance $d$, then $\lambda_i
m_{di}= (\lambda_i \E_i)_{uv}=(\A \E_i)_{uv}= \sum_{w \in
\Gamma_1(u) \cap \Gamma_{d-1}(v)} (\E_i)_{wv} = \delta
m_{d-1,i}$, hence $m_{d-1,i}=(-1)^i\frac {\pi_0}{n\pi_i}
\frac{\lambda_i}{\delta}$.
\end{proof}

An example of an almost distance-regular graph that illustrates
this proposition is the earlier mentioned graph F026A. It is
bipartite with $D=d=5$, hence it is $h$-punctually walk-regular
for $h=4,5$. Moreover, this graph is $2$-arc transitive, hence
it is also $2$-walk-regular ($h$-punctually walk-regular for
$h=0,1,2$). The intersection number $c_3$ is not well-defined
however, so the number of walks of length $3$ between vertices
at distance $3$ is not constant either, and therefore the graph
is not $3$-punctually walk-regular.

\subsection{From punctual to whole distance-regularity}

We already observed that Proposition \ref{prop escalar1 hdr }
and Corollary \ref{D-punc->drg} together imply the spectral
excess theorem. Proposition \ref{G is d-psr} shows that
$\omega_d a^{(d)}_d=1$, hence also Proposition \ref{pdr2cond}
implies the spectral excess theorem (again, with Corollary
\ref{D-punc->drg}). Finally, we will also show the connection
of Proposition \ref{propo h-punctual d-p} to this theorem. To
do this, we first restrict it to $h$-punctually
spectrum-regular graphs with spectrally maximum diameter.
\begin{propo}
\label{propo(b) h-punctual d-r} Let $h \le D$ and let $\G$ be
$h$-punctually spectrum-regular with spectrally maximum
diameter $D=d$. Then
$$
\overline{\delta}_h \le \frac 1n\left(\sum_{i=0}^d\frac{m_{hi}^2}{m_i}\right)^{-1}
$$
with equality if and only if $\G$ is $h$-punctually
distance-regular, in which case the crossed local
multiplicities are
$m_{hi}=\frac{p_h(\lambda_i)}{p_h(\lambda_0)}\frac{m_i}{n}$, $
i=0,\dots, d$.
\end{propo}

Notice that every (not necessarily regular) graph is
$0$-punctually distance-regular and $1$-punctually
distance-regular, because $\A_0=\I\in{\cal A}$ and
$\A_1=\A\in{\cal A}$. However, in general a graph is neither
$0$-punctually spectrum-regular nor $1$-punctually
spectrum-regular. If we apply Proposition~\ref{propo(b)
h-punctual d-r} for $h=0,1$ though, then we obtain reassuring
results. Indeed, if $\G$ is $0$-punctually spectrum-regular then
    $m_{0i}=\frac{m_i}{n}$, and
$$
\overline{\delta}_0=\frac 1n\left(\sum_{i=0}^d\frac{m_{0i}^2}{m_i}\right)^{-1}=
\frac 1n\left(\sum_{i=0}^d\frac{m_i}{n^2}\right)^{-1}=
n\left(\sum_{i=0}^d m_i\right)^{-1}=1.
$$
If $\G$ is $1$-punctually spectrum-regular then
    $m_{1i}=\frac{\lambda_i}{\lambda_0}\frac{m_i}n$ by
    Proposition \ref{1-puncw->walk}, and indeed
$$
\overline{\delta}_1
= \frac
1n\left(\sum_{i=0}^d\frac{m_i\lambda_i^2}{n^2\lambda_0^2}\right)^{-1}
 = n\lambda_0^2\left(\sum_{i=0}^d m_i\lambda_i^2\right)^{-1} =
n\lambda_0^2\left(n\lambda_0\right)^{-1} = \lambda_0.
$$

The most interesting result we obtain of course for $h=d~(=D)$. By
Proposition
    \ref{G is d-psr}, $\G$ is $d$-punctually
    spectrum-regular with $m_{di}=(-1)^i\frac
    {\pi_0}{n\pi_i}$. Then the condition of Proposition
    \ref{propo(b) h-punctual d-r} for $d$-punctual
    distance-regularity (and hence distance-regularity; we again use
    Corollary \ref{D-punc->drg}) becomes
$$
\overline{\delta}_d=\frac 1n\left(\sum_{i=0}^d\frac{m_{di}^2}{m_i}\right)^{-1}=
\frac 1n\left(\sum_{i=0}^d\frac{\pi_0^2}{n^2\pi_i^2m_i}\right)^{-1}=
\frac n{\pi_0^2}\left(\sum_{i=0}^d\frac 1{m_i\pi_i^2}\right)^{-1},
$$
which corresponds to the condition of the spectral excess theorem
for a (regular) graph to be distance-regular, as the right hand side
of the equation is known as an easy expression for $p_d(\lambda_0)$
in terms of the eigenvalues.\\

\noindent {\bf Acknowledgements} The authors would like to thank the referees for
their comments on an earlier version.

\end{document}